\documentclass[11pt]{article}
\usepackage{amssymb,amsmath,amsfonts,amsthm}

\usepackage{epsfig}

\numberwithin{equation}{section}

\newcommand{\R}{{\mathbb R}}

\newcommand{\intr}{\int\limits_{\R^N}}

\theoremstyle{plain}

\newtheorem{cor}{Corollary}

\newtheorem{theorem}{Theorem}[section]
\newtheorem{lemma}{Lemma}[section]

\theoremstyle{definition}
\newtheorem{defn}{Definition}

\begin{document}

\title{\vskip-0.3in Weighted Choquard equation perturbed with weighted nonlocal term}

\author{Gurpreet Singh\footnote{School of Mathematics, Trinity College Dublin; {\tt gurpreet.bajwa2506@gmail.com}}}

\maketitle

\begin{abstract}
We investigate the following problem
\begin{equation*}
-{\rm div}(v(x)|\nabla u|^{m-2}\nabla u)+V(x)|u|^{m-2}u= \Big(|x|^{-\theta}*\frac{|u|^{b}}{|x|^{\alpha}}\Big)\frac{|u|^{b-2}}{|x|^{\alpha}}u+\lambda\Big(|x|^{-\gamma}*\frac{|u|^{c}}{|x|^{\beta}}\Big)\frac{|u|^{c-2}}{|x|^{\beta}}u \quad\mbox{ in }\R^{N},
\end{equation*}
where $b, c, \alpha, \beta >0$, $\theta,\gamma \in (0,N)$, $N\geq 3$, $2\leq m< \infty$ and $\lambda \in \R$. Here, we are concerned with the existence of groundstate solutions and least energy sign-changing solutions and that will be done by using the minimization techniques on the associated Nehari manifold and the Nehari nodal set respectively.
\end{abstract}




\section{Introduction}\label{sec1}
In this paper, we study the problem
\begin{equation}\label{fr}
-{\rm div}(v(x)|\nabla u|^{m-2}\nabla u)+V(x)|u|^{m-2}u= \Big(|x|^{-\theta}*\frac{|u|^{b}}{|x|^{\alpha}}\Big)\frac{|u|^{b-2}}{|x|^{\alpha}}u+\lambda\Big(|x|^{-\gamma}*\frac{|u|^{c}}{|x|^{\beta}}\Big)\frac{|u|^{c-2}}{|x|^{\beta}}u \quad\mbox{ in }\R^{N},
\end{equation}
where $b, c, \alpha, \beta >0$, $\theta,\gamma \in (0,N)$, $2\leq m< \infty$, $N\geq 3$, $\lambda \in \R$ and ${\rm div}(v(x)|\nabla u|^{m-2}\nabla u)$ is the weighted $m$-Laplacian. Here $v$ is a Muckenhoupt weight and $|x|^{-\xi} $ is the Riesz potential of order $\xi \in (0, N)$. 
The function $V \in C(\R^{N})$ must satisfy either one or both of the following conditions:

\begin{enumerate}
\item[(V1)] $\inf_{\R^{N}}V(x)\geq V_{0}> 0$ ;
\item[(V2)] For all $M>0$ the set $\{x\in \R^N: V(x)\leq M\}$ has finite Lebesgue measure. 
\end{enumerate}

By taking $\lambda= 0$, the equation \eqref{fr} becomes the weighted Choquard equation driven by weighted $m$-Laplacian and is given by
\begin{equation}\label{squ}
-{\rm div}(v(x)|\nabla u|^{m-2}\nabla u)+V(x)|u|^{m-2}u= \Big(|x|^{-\theta}*\frac{|u|^{b}}{|x|^{\alpha}}\Big)\frac{|u|^{b-2}}{|x|^{\alpha}}u \quad\mbox{ in }\R^{N}.
\end{equation}
The case of $v(x)= V(x) \equiv 1$, $m=2$, $\theta= b= 2$ and $\alpha= 0$ in \eqref{squ} refers to the Choquard or nonlinear Schr\"odinger-Newton equation, that is,
\begin{equation}\label{ce}
-\Delta u+ u= (|x|^{-2}*u^2)u \quad\mbox{ in }\R^{N},
\end{equation}
and it was first studied by Pekar\cite{P1954} in 1954 for $N=3$. The equation \eqref{ce} had been used by Penrose in 1996 as a model in self-gravitating matter(see \cite{P1996}, \cite{P1998}). Also, if $v(x)\equiv 1$, $m= 2$ and $\alpha=\lambda=0$, then \eqref{squ} becomes stationary Choquard equation
$$
-\Delta u+ V(x)u= (|x|^{-\theta}*|u|^b)|u|^{b-2}u \quad\mbox{ in }\R^{N},
$$
which arises in quantum theory and in the theory of Bose-Einstein condensation. The Choquard equation has received a considerable attention in the last few decades and has been appeared in many different contexts and settings(see \cite{AFY2016, AGSY2017, GT2016, MV2017, MS2017, S2019}). In \cite{BY2020}, Benhamida and Yazidi investigated the critical Sobolev problem
\begin{equation}\label{ce1}
\left\{
\begin{aligned}
-{\rm div}(v(x)|\nabla u|^{m-2}\nabla u)&= |u|^{b^{*}-2}u + \lambda|u|^{c-2}u &&\quad\mbox{ in }\Omega,\\
u&> 0 &&\quad\mbox{ in }\Omega,\\
u&= 0 &&\quad\mbox{ on }\partial\Omega,
\end{aligned}
\right.
\end{equation}
where $\Omega \subset \R^N$ is a bounded domain, $N> b\geq 2$, $b\leq c< b^{*}$ and $b^{*}= \frac{Nb}{N-b}$ is called the critical Sobolev exponent. They investigated the existence of positive solutions which depends on the weight $v(x)$. In \cite{BN1983}, Brezis and Nirenberg studied the problem \eqref{ce1} for $v(x)\equiv 1$ and $m=2$ and it has stimulated a several work. The case $v\not \equiv$ constant and $m=2$ received a considerable attention and was considered by Hadiji and Yazidi in \cite{HY2007} for existence and nonexistence results, see also \cite{FS2016, HMPY2006}. 

\smallskip

In this article, we are interested in the groundstate solutions and least energy sign-changing solutions to \eqref{fr} and one could easily see that \eqref{fr} has a variational structure. To this aim, in the subsection below we provide variational framework and main results. 

\subsection{Variational Framework and Main Results}

\begin{defn}{ \it (Muckenhoupt Weight) }
Let $v\in \R^{N}$ be a locally integrable function such that $0<v<\infty$ a.e. in $\R^{N}$. Then $v\in A_m$, that is, the Muckenhoupt class if there exists a positive constant $C_{m, v}$ depending on $m$ and $v$ such that for all balls $B\in \R^{N}$, we have 
$$
\Big(\frac{1}{|B|}\int_{B}v dx \Big) \Big(\frac{1}{|B|}\int_{B}v^{-\frac{1}{m-1}} dx \Big)^{m-1}\leq C_{m, v}.
$$ 
\end{defn}

\begin{defn}{ \it (Weighted Sobolev Space) }
For any $v\in \R^{N}$, we denote the weighted Sobolev space by $W^{1, m}(\R^{N}, v)$ and is defined as 
$$
W^{1, m}(\R^{N}, v)= \{u: \R^{N}\to \R \;\; {\rm measurable}: \;\; ||u||_{1,m,v}< \infty \},
$$
with respect to the norm
\begin{equation}\label{w1}
||u||_{1,m,v}= \Big( \int_{\R^N}|u(x)|^{m}v(x) dx +\int_{\R^N}|\nabla u|^{m}v(x) dx \Big)^{\frac{1}{m}}.
\end{equation}
And the space $X=  W^{1, m}_{0}(\R^{N}, v)$ is the closure of $(C_{c}^{\infty}(\R^{N}), ||.||_{1,m,v})$ with respect to the norm
\begin{equation}\label{w1}
||u||_{X}= \Big(\int_{\R^N}|\nabla u|^{m}v(x) dx \Big)^{\frac{1}{m}}.
\end{equation}
\end{defn}

\begin{defn}{ \it (Subclass of $A_m$) }
Let us denote the subclass of $A_m$ by $A_p$ and define $A_p$ as
$$
A_p= \Big\{ v\in A_m: \;\; v^{-p}\in L^{1}(\R^{N}) \quad{\rm for \; some }\;\; p\in [\frac{1}{m-1}, \infty)\cap \big(\frac{N}{m}, \infty \big) \Big\}.
$$
\end{defn}

\begin{defn}{\it (Weighted Morrey space) }
Assume $1< m< \infty$, 	$r> 0$ and $v\in A_m$. Then $u\in L^{m, r}(\R^{N}, v)$- the weighted Morrey space, if $u\in L^{m}(\R^{N}, v)$, where 
$$
L^{m}(\R^{N}, v)= \Big\{u: \R^{N}\to \R \;\; measurable: \;\; \int_{\R^N}v(x)|u|^m dx< \infty \Big\},
$$
and 
$$
||u||_{L^{m, r}(\R^{N}, v)}= \sup_{x\in \R^{N}, R> 0}\Big(L\int_{B(x, R)}v(y)|u(y)|^m dy \Big)^{\frac{1}{m}}< \infty,
$$
where $L= \frac{R^r}{\int_{B(x, R)}v(x) dx}$ and $B(x, R)$ is the ball centered at $x$ and radius $R$.
\end{defn}

Throughout this paper, we have the following assumption on the weight function $v(x)$:
\begin{itemize}
\item For $1\leq m_p\leq N$, $v\in A_p$ and 
$$
\frac{1}{v}\in L^{t, mN-\eta t(m-1)}(\R^N, v),
$$
where $t> N$ and $0< \eta< {\rm min}\{1, \frac{mN}{t(m-1)}\}$.
\end{itemize}

\medskip

Next, let us define the functional space
$$
X_{v}(\R^{N})= \Big\{u\in X: \int_{\R^{N}}V(x)|u|^m< \infty \Big\},
$$
endowed with the norm
$$
\|u\|_{X_{v}}=\Big[\int_{\R^N}v(x)|\nabla u|^m +\int_{\R^{N}}V(x)|u|^m \Big]^{\frac{1}{m}}.
$$
Throughout this paper, assume that $b$ satisfies 
\begin{equation}\label{p}
\frac{mp(2N-2\alpha-\theta)}{2N(p+1)}< b< \frac{mp(2N-2\alpha-\theta)}{2N+2p(N-m)},
\end{equation}
or
\begin{equation}\label{p1}
\frac{2N-2\alpha-\theta}{2N}< b< \infty,
\end{equation}
and $c$ satisfies
\begin{equation}\label{q}
\frac{mp(2N-2\beta-\gamma)}{2N(p+1)}< c< \frac{mp(2N-2\beta-\gamma)}{2N+2p(N-m)},
\end{equation}
or
\begin{equation}\label{q1}
\frac{2N-2\beta-\gamma}{2N}< c< \infty.
\end{equation} 

We also need the following double weighted Hardy-Littlewood-Sobolev inequality by Stein and Weiss(see \cite{SW1958})
\begin{equation}\label{hli}
\Big| \int_{\R^N} \Big(|x|^{-\delta}*\frac{u}{|x|^{\mu}}\Big)\frac{v}{|x|^{\mu}} \Big| \leq C\|u\|_p \|v\|_q,	
\end{equation}
for $\delta \in (0, N)$, $\mu \geq 0$, $u\in L^{p}(\R^N)$ and $v\in L^{q}(\R^N)$ such that
$$
1-\frac{1}{q}- \frac{\delta}{N}< \frac{\mu}{N}< 1-\frac{1}{q} \quad{ and }\;\;  \frac{1}{p}+\frac{1}{q}+\frac{\delta+2 \mu}{N}= 2.
$$

Define the energy functional 
${\mathcal L}_\lambda:X_{v}(\R^{N}) \rightarrow \R$ by
\begin{equation}\label{fr1}
\begin{aligned}
{\mathcal L}_\lambda(u)&=\frac{1}{m}\|u\|_{X_{v}}^{m}-\frac{1}{2b}\int_{\R^{N}}\Big({|x|^{-\theta}*\frac{|u|^{b}}{|x|^{\alpha}}}\Big)\frac{|u|^{b}}{|x|^{\alpha}}-\frac{\lambda}{2c}\int_{\R^{N}}\Big(|x|^{-\gamma}*\frac{|u|^{c}}{|x|^{\beta}}\Big)\frac{|u|^{c}}{|x|^{\beta}}.
\end{aligned}
\end{equation}
The energy functional ${\mathcal L}_\lambda$ is well defined by using \eqref{p} to \eqref{q1} together with the double weighted Hardy-Littlewood-Sobolev inequality \eqref{hli} and  moreover ${\mathcal L}_\lambda\in C^1(X_v)$. Any solution of \eqref{fr} is a critical point of the energy functional  ${\mathcal L}_\lambda$.
We first deal with the existence of groundstate solutions for the equation \eqref{fr} under the assumption that $V$ satisfies $(V1)$. To this aim, we shall be using a minimization method on the Nehari manifold associated with ${\mathcal L}_\lambda$, which is defined as 
\begin{equation}\label{nm}
{\mathcal N_\lambda}=\{u\in X_{v}(\R^N)\setminus\{0\}: \langle {\mathcal L}_\lambda'(u),u\rangle=0\}.
\end{equation}

The groundstate solutions will be obtained as minimizers of
$$
d_{\lambda}=\inf_{u\in {\mathcal N_\lambda}}{\mathcal L}_\lambda(u).
$$
Our main result regarding the existence of groundstate solutions is stated below.
\begin{theorem}\label{gstate}
Assume $N> m\geq 2$, $b>c> \frac{m}{2}$,  $\lambda> 0$, $\theta+2\alpha< N$,$\gamma+2\beta< N$. If $b$, $c$ satisfies \eqref{p} and \eqref{q} or if $b$, $c$ satisfy \eqref{p1} and \eqref{q1} and $V$ satisfies $(V1)$, then the equation \eqref{fr} has a groundstate solution $u\in X_{v}(\R^{N})$. 
\end{theorem}
We rely on the analysis of the Palais-Smale sequences for ${\mathcal L}_\lambda \!\mid_{\mathcal N_\lambda}$. We will show that any Palais-Smale sequence of ${\mathcal L}_\lambda \!\mid_{\mathcal N_\lambda}$ is either converging strongly to its weak limit or differs from it by a finite number of sequences, which are the translated solutions of \eqref{squ} by using ideas from \cite{CM2016,CV2010}. Here, we shall be relying on several weighted nonlocal Brezis-Lieb results which we have presented in Section $2$.

\medskip

Next, we study the least energy sign-changing solutions of \eqref{fr}. Now, we need $V$ to satisfy both the conditions $(V1)$ and $(V2)$. We use the minimization method on the Nehari nodal set defined as 
$$
{\cal \overline {N}_\lambda}= \Big\{u\in  X_{v}(\R^{N}): u^{\pm} \neq 0 \mbox{ and } \langle {\mathcal L}_\lambda'(u),u^{\pm}\rangle  \mbox{ = } 0 \Big\},
$$
and solutions will be obtained as minimizers for 
$$
\overline{d_\lambda}= \inf_{u\in {\cal \overline{N}}_{\lambda}}{\mathcal L}_\lambda(u).
$$
Here, we have
\begin{equation*}
\begin{aligned}
\langle {\mathcal L}_\lambda'(u),u^{\pm} \rangle& = \|u^{\pm}\|_{X_{v}}^{m}-\int_{\R^{N}}\Big(|x|^{-\theta}*\frac{(u^{\pm})^{b}}{|x|^\alpha}\Big)\frac{(u^{\pm})^{b}}{|x|^\alpha}-\lambda\int_{\R^{N}}\Big(|x|^{-\gamma}*\frac{(u^{\pm})^{c}}{|x|^\beta}\Big)\frac{(u^{\pm})^{c}}{|x|^\beta}\\
&- \int_{\R^{N}}\Big(|x|^{-\theta}*\frac{(u^{\pm})^{b}}{|x|^\alpha}\Big)\frac{(u^{\mp})^{b}}{|x|^\alpha}-\lambda\int_{\R^{N}}\Big(|x|^{-\gamma}*\frac{(u^{\pm})^{c}}{|x|^beta}\Big)\frac{(u^{\mp})^{c}}{|x|^\beta}.
\end{aligned}
\end{equation*}
 
\medskip
 
We now state our second main result in reference to the least energy sign-changing solutions.
\begin{theorem}\label{signc}
Let $N> m\geq 2$, $b>c> m$,  $\lambda \in \R$, $\theta+2\alpha< m$,$\gamma+2\beta< m$. If $b$, $c$ satisfies \eqref{p} and \eqref{q} or if $b$, $c$ satisfy \eqref{p1} and \eqref{q1} and $V$ satisfies both $(V1)$ and $(V1)$, then the equation \eqref{fr} has a least energy sign-changing solution $u\in X_{v}(\R^{N})$.
\end{theorem}

\medskip

Rest of the paper is organized as follows. In Section $2$ we collect some preliminary results on Sobolev embeddings and  weighted nonlocal versions of the Brezis-Lieb lemma which will be crucial to our investigation of groundstate solutions of \eqref{fr}. Section $3$ and $4$ consists of the proofs of our main results.

\bigskip

\section{Preliminary results}

\begin{lemma}\label{se}(\cite{AA2011}, \cite{DKN1997}, \cite{G2018})
For any $v\in A_p$, the inclusion map
$$
X_v \hookrightarrow W_{0}^{1, m_p}(\R^N) \hookrightarrow \left\{ 
\begin{aligned}
L^{s}(\R^N), & \quad{ for }\;\; m_p\leq s\leq m_{p}^{*}, && \quad{ when }\;\; 1\leq m_p< N,\\
L^{s}(\R^N), & \quad{ for }\;\; 1\leq s< \infty, && \quad{ when }\;\; m_p= N,
\end{aligned} 
\right.
$$
is continuous, where $m_p= \frac{mp}{p+1}$ and $m_{p}^{*}= \frac{Nm_p}{N-m_p}$. Here, $m_{p}^{*}$ is called the critical Sobolev exponent. Moreover, the embeddings are compact except when $s=m_p^{*}$ in case of $1\leq m_p< N$.
\end{lemma}

\begin{lemma}\label{cc}(\cite[Lemma 1.1]{L1984}, \cite[Lemma 2.3]{MV2013})
There exists a constant $C>0$ such that for any $u\in X_{v}(\R^{N})$ we have
$$
\int_{\R^N}|u|^r \leq C||u||\Big(\sup_{y\in \R^N} \int_{B_1(y)}|u|^r \Big)^{1-\frac{2}{r}},
$$
where $r\in [m_p,m_{p}^*]$.
\end{lemma}

\begin{lemma}\label{bogachev}(\cite[Proposition 4.7.12]{B2007})
Let $(z_n)$ be a bounded sequence in $L^r(\R^N)$ for some $r\in (1,\infty)$ which converges to $z$ almost everywhere. Then $w_n\rightharpoonup w$ weakly in $L^r(\R^N)$.
\end{lemma}

\begin{lemma}\label{blc}{\sf (\it Local Brezis-Lieb lemma)}
Let $(z_n)$ be a bounded sequence in $L^r(\R^N)$ for some $r\in (1,\infty)$ which converges to $z$ almost everywhere. Then, for every $q\in [1,r]$ we have
$$
\lim_{n\to\infty}\int_{\R^N}\big| |z_n|^q-|z_n-z|^q-|z|^q\big|^{\frac{r}{q}}=0\,,
$$
and
$$
\lim_{n\to\infty}\int_{\R^N}\big| |z_n|^{q-1}z_n-|z_n-z|^{q-1}(z_n-z)-|z|^{q-1}z\big|^{\frac{r}{q}}=0.
$$
\end{lemma}
\begin{proof}
Let us fix $\varepsilon>0$, then there exists $C(\varepsilon)>0$ such that for all $g$,$h\in \R$ we have
\begin{equation}\label{lb1}
\Big||g+h|^{q}-|g|^{q}\big|^{\frac{r}{q}}\leq \varepsilon|g|^{r}+C(\varepsilon)|h|^{r}.
\end{equation}
By equation \eqref{lb1}, one could obtain
$$
\begin{aligned}
|f_{n, \varepsilon}|=& \Big( \Big||z_{n}|^{q}-|z_{n}-z|^{q}-|z^{q}|\Big|^{\frac{r}{q}}-\varepsilon|z_{n}-z|^{r}\Big)^{+}\\
&\leq (1+C(\varepsilon))|z|^{r}.
\end{aligned}
$$
Next, by Lebesgue Dominated Convergence theorem, we get
\begin{equation}
\intr {f_{n, \varepsilon}} \rightarrow 0  \quad\mbox{ as } n\rightarrow \infty.
\end{equation}
Hence, we deduce that
$$
\Big||z_{n}|^{q}-|z_{n}-z|^{q}-|z|^{q}\Big|^{\frac{r}{q}}\leq f_{n, \varepsilon}+\varepsilon|z_{n}-z|^{r},
$$
and this further gives 
$$
\limsup_{n\rightarrow \infty} {\intr \Big||z_{n}|^{q}-|z_{n}-z|^{q}-|z|^{q}\Big|^{\frac{r}{q}}}\leq c\varepsilon,
$$
where $c= \sup_{n}|z_{n}-z|_{r}^{r}< \infty$. In order to conclude our proof, we let $\varepsilon \rightarrow 0$.
\end{proof}

\begin{lemma}\label{nlocbl}{\sf (\it Weighted Nonlocal Brezis-Lieb lemma}(\cite[Lemma 2.4]{MV2013})
Let $N\geq 3$, $\alpha\geq 0$, $\theta\in (0,N)$, $\theta+2\alpha< N$ and $b\in [1,\frac{2N}{2N-2\alpha-\theta})$. Assume $(u_n)$ is a bounded sequence in $L^{\frac{2Nb}{2N-2\alpha-\theta}}(\R^N)$ such that $u_n \rightarrow u$ almost everywhere in $\R^N$. Then
$$
\int_{\R^N}\Big(|x|^{-\theta}*\frac{|u_n|^{b}}{|x|^{\alpha}}\Big)\frac{|u_n|^b}{|x|^\alpha}dx-\int_{\R^N}\Big(|x|^{-\theta}*\frac{|u_n-u|^b}{|x|^{\alpha}}\Big)\frac{|u_n-u|^b}{|x|^\alpha}dx \rightarrow \int_{\R^N}\Big(|x|^{-\theta}*\frac{|u|^b}{|x|^\alpha}\Big)\frac{|u|^b}{|x|^\alpha}dx.
$$
\end{lemma}
\begin{proof} For $n\in N$, we notice that
\begin{equation}\label{nb1}
\begin{aligned}
&\int_{\R^N}\Big(|x|^{-\theta}*\frac{|u_n|^{b}}{|x|^{\alpha}}\Big)\frac{|u_n|^b}{|x|^\alpha}dx-\int_{\R^N}\Big(|x|^{-\theta}*\frac{|u_n-u|^b}{|x|^{\alpha}}\Big)\frac{|u_n-u|^b}{|x|^\alpha}dx\\&=\intr \Big[|x|^{-\theta}*\Big(\frac{1}{|x|^\alpha}|u_n|^b-\frac{1}{|x|^\alpha}|u_n-u|^b\Big)\Big]\Big(\frac{1}{|x|^\alpha}|u_n|^b-\frac{1}{|x|^\alpha}|u_n-u|^b\Big)dx\\
&+2\intr \Big[|x|^{-\theta}*\Big(\frac{1}{|x|^\alpha}|u_n|^b-\frac{1}{|x|^\alpha}|u_n-u|^b\Big)\Big]\frac{1}{|x|^\alpha}|u_n-u|^b dx.
\end{aligned}
\end{equation}
Next, we use Lemma \ref{blc} with $q=b$, $r=\frac{2Nb}{2N-2\alpha-\theta}$ to get $|u_n-u|^b-|u_n|^b\to |u|^b$ strongly in $L^{\frac{2N}{2N-2\alpha-\theta}}(\R^N)$ and by Lemma \ref{bogachev} we have $|u_n-u|^{b}\rightharpoonup 0$ weakly in $L^{\frac{2N}{2N-2\alpha-\theta}}(\R^{N})$. Also by the double weighted Hardy-Littlewood-Sobolev inequality \eqref{hli} we obtain 
$$
|x|^{-\theta}*\Big(\frac{1}{|x|^\alpha}|u_n-u|^b-\frac{1}{|x|^\alpha}|u_n|^{b}\Big)\to |x|^{-\theta}*\frac{|u|^b}{|x|^\alpha} \quad\mbox{ in } L^{\frac{2N}{\theta+2\alpha}}(\R^N).
$$
Using all the above arguments and passing to the limit in \eqref{nb1} we conclude the proof.
	
\end{proof}

\begin{lemma}\label{anbl}
Let $N\geq 3$, $\alpha\geq 0$, $\theta\in (0,N)$, $\theta+2\alpha< N$ and $b\in [1,\frac{2N}{2N-2\alpha-\theta})$. Assume $(u_n)$ is a bounded sequence in $L^{\frac{2Nb}{2N-2\alpha-\theta}}(\R^N)$ such that $u_n \rightarrow u$ almost everywhere in $\R^N$. Then, for any $h\in L^{\frac{2Nb}{2N-2\alpha-\theta}}(\R^N)$ we have
$$
\int_{\R^N}\Big(|x|^{-\theta}*\frac{|u_n|^b}{|x|^\alpha}\Big)\frac{1}{|x|^\alpha}|u_n|^{b-2}u_nh\; dx \rightarrow \int_{\R^N} \Big(|x|^{-\theta}*\frac{|u|^b}{|x|^\alpha}\Big)\frac{1}{|x|^\alpha}|u|^{b-2}uh\; dx.
$$
\end{lemma}
\begin{proof} 
Say $h=h^+-h^-$, then it is enough to prove our lemma for $h\geq 0$. Let $v_n=u_n-u$ and notice that
\begin{equation}\label{bl1}
\begin{aligned}
\intr \Big(|x|^{-\theta}*\frac{|u_n|^b}{|x|^{\alpha}}\Big)\frac{1}{|x|^\alpha}|u_n|^{b-2}u_nh=& \intr \Big[|x|^{-\theta}*\Big(\frac{1}{|x|^\alpha}|u_n|^b-\frac{1}{|x|^\alpha}|v_n|^b\Big)\Big]\Big(\frac{1}{|x|^\alpha}|u_n|^{b-2}u_nh-\frac{1}{|x|^\alpha}|v_n|^{b-2}v_nh\Big)\\
&+\intr \Big[|x|^{-\theta}*\Big(\frac{1}{|x|^\alpha}|u_n|^b-\frac{1}{|x|^\alpha}|v_n|^b\Big)\Big]\frac{1}{|x|^\alpha}|v_n|^{b-2}v_nh\\
&+\intr \Big[|x|^{-\theta}*\Big(\frac{1}{|x|^\alpha}|u_n|^{b-2}u_nh-\frac{1}{|x|^\alpha}|v_n|^{b-2}v_n h\Big)\Big]\frac{|v_n|^b}{|x|^\alpha}\\
&+\intr \Big(|x|^{-\theta}*\frac{|v_n|^b}{|x|^\alpha}\Big)\frac{1}{|x|^\alpha}|v_n|^{p-2}v_nh.
\end{aligned}
\end{equation}
Now, apply Lemma \ref{blc} with $q=b$ and $r=\frac{2Nb}{2N-2\alpha-\theta}$ and by taking $(z_n,z)=(u_n,u)$ and then $(z_n,z)=(u_nh^{1/b}, u h^{1/b})$ respectively, we get
$$
\left\{
\begin{aligned}
&|u_n|^b-|v_n|^b\to |u|^b \\
&|u_n|^{b-2}u_nh-|v_n|^{b-2}v_nh\to |u|^{b-2}uh
\end{aligned}
\right.
\quad\mbox{ strongly in }\; L^{\frac{2N}{2N-2\alpha-\theta}}(\R^N).
$$
Further, using the double weighted Hardy-Littlewood-Sobolev inequality we obtain
\begin{equation}\label{est00}
\left\{
\begin{aligned}
&|x|^{-\theta}*\Big(\frac{1}{|x|^\alpha}|u_n|^b-\frac{1}{|x|^\alpha}|v_n|^b\Big)\to |x|^{-\theta}*\frac{|u|^b}{|x|^\alpha} \\
&|x|^{-\theta}*\Big(\frac{1}{|x|^\alpha}|u_n|^{b-2}u_nh-\frac{1}{|x|^\alpha}|v_n|^{b-2}v_nh\Big)\to |x|^{-\theta}*\Big(\frac{1}{|x|^\alpha}|u|^{p-2}uh\Big)
\end{aligned}
\right.
\quad\mbox{ strongly in }\; L^{\frac{2N}{\theta+2\alpha}}(\R^N).
\end{equation}
By Lemma \ref{bogachev} we have
\begin{equation}\label{est01}
\left\{
\begin{aligned}
&|u_n|^{b-2}u_n h\rightharpoonup |u|^{b-2}uh\\ &|v_n|^b\rightharpoonup 0\\
&|v_n|^{b-2}v_nh\rightharpoonup 0
\end{aligned}
\right.
\quad \mbox{ weakly in }\; L^{\frac{2N}{2N-2\alpha-\theta}}(\R^N)
\end{equation}
From \eqref{est00} and \eqref{est01} we get
\begin{equation}\label{est02}
\begin{aligned}
&\intr \Big[|x|^{-\theta}*\Big(\frac{1}{|x|^\alpha}|u_n|^b-\frac{1}{|x|^\alpha}|v_n|^b\Big)\Big]\Big(\frac{1}{|x|^\alpha}|u_n|^{b-2}u_nh-\frac{1}{|x|^\alpha}|v_n|^{b-2}v_n h\Big)\to \int_{\R^N} \Big(|x|^{-\theta}*\frac{|u|^b}{|x|^\alpha}\Big)\frac{1}{|x|^\alpha}|u|^{b-2}uh,\\
& \intr \Big[|x|^{-\theta}*\Big(\frac{1}{|x|^\alpha}|u_n|^b-\frac{1}{|x|^\alpha}|v_n|^b\Big)\Big]\frac{1}{|x|^\alpha}|v_n|^{b-2}v_nh\to 0,\\
&\intr \Big[|x|^{-\theta}*\Big(\frac{1}{|x|^\alpha}|u_n|^{b-2}u_nh-\frac{1}{|x|^\alpha}|v_n|^{b-2}v_nh\Big)\Big]\frac{|v_n|^b}{|x|^\alpha}\to 0.
\end{aligned}
\end{equation}
Using H\"older's inequality and the double weighted Hardy-Littlewood-Sobolev inequality, we have
\begin{equation}\label{est03}
\begin{aligned}
\left| \intr \Big(|x|^{-\theta}*\frac{|v_n|^{b}}{|x|^\alpha}\Big)\frac{1}{|x|^\alpha}|v_n|^{b-2}v_nh \right|& \leq \|v_n\|^b_{\frac{2Nb}{2N-2\alpha-\theta}}\||v_n|^{b-1}h\|_{\frac{2N}{2N-2\alpha-\theta}}\\
&\leq C \||v_n|^{b-1}h\|_{\frac{2N}{2N-2\alpha-\theta}}.
\end{aligned}
\end{equation}
Also, by Lemma \ref{bogachev} we have $v_n^{\frac{2N(b-1)}{2N-2\alpha-\theta}}\rightharpoonup 0$ weakly in $L^{\frac{b}{b-1}}(\R^N)$ so 
$$
\||v_n|^{b-1}h\|_{\frac{2N}{2N-2\alpha-\theta}}=\left(\intr |v_n|^{\frac{2N(b-1)}{2N-2\alpha-\theta}}|h|^{\frac{2N}{2N-2\alpha-\theta}}  \right)^{\frac{2N-2\alpha-\theta}{2N}}\to 0.
$$
Hence, by \eqref{est03} we have
\begin{equation}\label{est04}
\lim_{n\to \infty} \intr \Big(|x|^{-\theta}*\frac{|v_n|^{b}}{|x|^\alpha}\Big)\frac{1}{|x|^\alpha}|v_n|^{b-2}v_nh=0.
\end{equation}
	Passing to the limit in \eqref{bl1}, from \eqref{est02} and \eqref{est04} we reach the conclusion.
\end{proof}

\bigskip

In the next section, we investigate the groundstate solutions to \eqref{fr}.
\section{Proof of Theorem \ref{gstate}}
Assume $\lambda> 0$.  For $u,\phi\in X_{v}(\R^{N})$  we have
\begin{equation*}
\begin{aligned}
\langle {\mathcal L}'_\lambda(u), \phi \rangle&= \int_{\R^{N}} v(x) |\nabla u|^{m-2}\nabla u \nabla \phi + \int_{\R^{N}} V(x)|u|^{m-2}u \phi -\int_{\R^{N}}\Big({|x|^{-\theta}*\frac{|u|^{b}}{|x|^{\alpha}}}\Big)\frac{|u|^{b-1}}{|x|^{\alpha}}\phi \\ &-\lambda\int_{\R^{N}}\Big(|x|^{-\gamma}*\frac{|u|^{c}}{|x|^{\beta}}\Big)\frac{|u|^{c-1}}{|x|^{\beta}}\phi.
\end{aligned}
\end{equation*}
Also, for $t>0$ we have
\begin{equation*}
\begin{aligned}
\langle {\mathcal L}'_\lambda(tu), tu \rangle &= t^{m}\|u\|_{X_{v}}^{m}- t^{2b}\int_{\R^{N}}\Big({|x|^{-\theta}*\frac{|u|^{b}}{|x|^{\alpha}}}\Big)\frac{|u|^{b}}{|x|^{\alpha}}-\lambda t^{2c}\int_{\R^{N}}\Big(|x|^{-\gamma}*\frac{|u|^{c}}{|x|^{\beta}}\Big)\frac{|u|^{c}}{|x|^{\beta}}.
\end{aligned}
\end{equation*}

As $b>c>\frac{m}{2}$, so the equation $\langle {\mathcal L}_\lambda'(tu),tu \rangle= 0 $ has a unique positive solution $t=t(u)$. The element $tu\in {\mathcal N_\lambda}$ is called the {\it projection of $u$} on ${\mathcal N_\lambda}$. The main properties of the Nehari manifold ${\mathcal N_\lambda}$ which we use in this paper are given by the following lemmas:

\begin{lemma}\label{nehari1}
${\mathcal L}_\lambda \!\mid_{\mathcal N_\lambda}$ is coercive and bounded from below by a positive constant.
\end{lemma}
\begin{proof} 
First we show that ${\mathcal L}_\lambda \!\mid_{\mathcal N_\lambda}$ is coercive. Note that
$$
\begin{aligned}
{\mathcal L}_\lambda(u)&= {\mathcal L}_\lambda(u)-\frac{1}{2c}\langle {\mathcal L}_\lambda'(u), u \rangle \\
&=\Big(\frac{1}{m}-\frac{1}{2c}\Big)\|u\|_{X_{v}}^m+\Big(\frac{1}{2c}-\frac{1}{2b}\Big)\int_{\R^{N}} \Big(|x|^{-\theta}*\frac{|u|^{b}}{|x|^{\alpha}}\Big)\frac{|u|^{b}}{|x|^\alpha} \\
&\geq \Big(\frac{1}{m}-\frac{1}{2c}\Big)\|u\|_{X_{v}}^m.
\end{aligned}
$$
Next, using the double weighted Hardy-Littlewood-Sobolev inequality together with the continuous embeddings  $X_{v}(\R^N) \hookrightarrow L^{\frac{2Nb}{2N-2\alpha-\theta}}(\R^N)$ and  $X_{v}(\R^N) \hookrightarrow L^{\frac{2Nc}{2N-2\beta-\gamma}}(\R^N)$, for any $u\in {\mathcal N_\lambda}$ we have
$$
\begin{aligned}
0=\langle {\mathcal L}_\lambda'(u),u\rangle& =\|u\|_{X_{v}}^m-\int_{\R^{N}}\Big(|x|^{-\theta}*\frac{|u|^{b}}{|x|^\alpha}\Big)\frac{|u|^{b}}{|x|^\alpha}-\lambda \int_{\R^{N}}\Big(|x|^{-\gamma}*\frac{|u|^{c}}{|x|^\beta}\Big)\frac{|u|^{q}}{|x|^\beta}\\
&\geq \|u\|_{X_{v}}^m-C\|u\|_{X_{v}}^{2b}-C_\lambda \|u\|_{X_{v}}^{2c}.
\end{aligned}
$$
Therefore, there exists $C_0>0$ such that
\begin{equation}\label{cnot}
\|u\|_{X_{v}}\geq C_0>0\quad\mbox{for all }u\in {\mathcal N_\lambda}.
\end{equation}
Hence, using coercivity of ${\mathcal L}_\lambda \!\mid_{\mathcal N_\lambda}$ and \eqref{cnot}, we get
$$
{\mathcal L}_\lambda(u) \geq \Big(\frac{1}{m}-\frac{1}{2c}\Big) C_0^m>0.
$$
\end{proof}
	
\begin{lemma}\label{nehari2}
Any critical point $u$ of ${\mathcal L}_\lambda \!\mid_{\mathcal N_\lambda}$ is a free critical point.
\end{lemma}
\begin{proof}
Let us assume $ {\mathcal K}(u)=\langle {\mathcal L}_\lambda'(u),u\rangle $ for any $u \in X_{v}(\R^N)$. Using \eqref{cnot}, for any $u \in X_{v}(\R^N)$  we get
\begin{equation}\label{cnot1}
\begin{aligned}
\langle {\mathcal K}'(u),u\rangle&=m\|u\|^m-2b\int_{\R^{N}}\Big(|x|^{-\theta}*\frac{|u|^{b}}{|x|^\alpha}\Big)\frac{|u|^{b}}{|x|^\alpha}-2c \lambda\int_{\R^{N}}\Big(|x|^{-\gamma}*\frac{|u|^{c}}{|x|^\beta}\Big)\frac{|u|^{c}}{|x|^{\beta}}\\
&=(m-2c)\|u\|_{X_{v}}^m-2(b-c)\int_{\R^{N}}\Big(|x|^{-\theta}*\frac{|u|^{b}}{|x|^\alpha}\Big)\frac{|u|^{b}}{|x|^\alpha}\\
&\leq -(2c-m)\|u\|_{X_{v}}^m\\
&<-(2c-m)C_0.
\end{aligned}
\end{equation}
Now, say $u\in {\mathcal N_\lambda}$ is a critical point of ${\mathcal L}_\lambda \!\mid_{\mathcal N_\lambda}$. Using the Lagrange multiplier theorem, there exists $\nu \in \R$ such that ${\mathcal L}_\lambda'(u)=\nu {\mathcal K}'(u)$. So, in particular we have $\langle {\mathcal L}_\lambda '(u),u\rangle=\nu \langle {\mathcal K}'(u),u\rangle$. Since $\langle {\mathcal K}'(u),u\rangle<0$, which further implies $\nu=0$ so ${\mathcal L}_\lambda '(u)=0$.
	
\end{proof}

\begin{lemma}\label{nehari3}
Any sequence $(u_n)$ which is a $(PS)$ sequence for ${\mathcal L}_\lambda \!\mid_{\mathcal N_\lambda}$ is a $(PS)$ sequence for ${\mathcal L}_\lambda$.
\end{lemma}
\begin{proof}
Assume that $(u_n)\subset {\mathcal N_\lambda}$ is a $(PS)$ sequence for ${\mathcal L}_\lambda \!\mid_{\mathcal N_\lambda}$. As,
$$
{\mathcal L}_\lambda(u_n)\geq \Big(\frac{1}{m}-\frac{1}{2c}\Big)\|u_n\|_{X_{v}}^m,
$$
this gives us that $(u_n)$ is bounded in ${X_{v}}$. Next, we show that ${\mathcal L}'_\lambda(u_n)\to 0$. Since,
$$
{\mathcal L}'_\lambda(u_n)- \nu_n {\mathcal K}'(u_n)= {\mathcal L}'_\lambda \!\mid_{\mathcal N_\lambda}(u_n)= o(1),
$$
for some $\nu_n \in \R$. Hence,
$$
\nu_n \langle {\mathcal K}'(u_n),u_n \rangle= \langle {\mathcal L}_\lambda '(u_n),u_n \rangle + o(1)= o(1).
$$
Using \eqref{cnot1}, we get $\nu_n \to 0$ which further gives us that ${\mathcal L}_\lambda '(u_n) \to 0$.
\end{proof}

\subsection{Compactness result}\label{compc}
Define the energy functional ${\mathcal I}:X_{v}(\R^N)\to \R$ by
$$
{\mathcal I}(u)=\frac{1}{m}\|u\|^{m}-\frac{1}{2b}\intr \Big(|x|^{-\theta}*\frac{|u|^b}{|x|^\alpha}\Big)\frac{|u|^b}{|x|^\alpha},
$$
and the associated Nehari manifold for ${\mathcal I}$ is given as
$$
{\mathcal N}_{{\mathcal I}}=\{u\in X_{v}(\R^N)\setminus\{0\}: \langle {\mathcal I}'(u),u\rangle=0\},
$$
and let 
$$
d_{\mathcal I}=\inf_{u\in {\mathcal N}_{\mathcal I}}{\mathcal I}(u).
$$

Also, for all $\phi \in C^{\infty}_{0}(\R^N)$, we have
\begin{equation*}
\begin{aligned}
\langle {\mathcal I}'(u), \phi \rangle&= \int_{\R^{N}} v(x) |\nabla u|^{m-2}\nabla u \nabla \phi + \int_{\R^{N}} V(x)|u|^{m-2}u \phi -\int_{\R^{N}}\Big({|x|^{-\theta}*\frac{|u|^{b}}{|x|^{\alpha}}}\Big)\frac{|u|^{b-1}}{|x|^{\alpha}}\phi.
\end{aligned}
\end{equation*}
and
\begin{equation*}
\langle {\mathcal I}'(u),u\rangle= \|u\|_{X_{v}}^{m}-\int_{\R^{N}}\Big(|x|^{-\theta}*\frac{|u|^{b}}{|x|^{\alpha}}\Big)\frac{|u|^{b}}{|x|^{\alpha}}.
\end{equation*}

\begin{lemma}\label{compact}
Let us assume that $(u_n)\subset{\mathcal N}_{\mathcal I}$ is a $(PS)$ sequence of ${\mathcal L}_\lambda \!\mid_{{\mathcal N}_{\lambda}}$, that is,
\begin{enumerate}
\item[(a)] $({\mathcal L}_\lambda(u_n))$ is bounded;
\item[(b)] ${\mathcal L}_\lambda'\!\mid_{{\mathcal N}_{\lambda}}(u_n)\to 0$ strongly in $X_{v}^{-1}(\R^N)$.
\end{enumerate}
Then there exists a solution $u\in X_{v}(\R^N)$ of \eqref{fr} such that, if we replace the sequence $(u_n)$ with a subsequence, then one of the following alternative holds:
	
\smallskip
	
\noindent $(A_1)$ either $u_n\to u$ strongly in $X_{v}(\R^N)$;
	
or
	
\smallskip
	
\noindent $(A_2)$ $u_n\rightharpoonup u$ weakly in $X_{v}(\R^N)$ and there exists a positive integer $k\geq 1$ and $k$ functions $u_1,u_2,\dots, u_k\in X_{v}(\R^N)$ which are nontrivial  weak solutions to \eqref{squ} and $k$ sequences of points $(w_{n,1})$, $(w_{n,2})$, $\dots$, $(w_{n,k})\subset \R^N$ such that the following conditions hold:
\begin{enumerate}
\item[(i)] $|w_{n,j}|\to \infty$ and $|w_{n,j}-w_{n,i}|\to \infty$  if $i\neq j$, $n\to \infty$;
\item[(ii)] $ u_n-\sum_{j=1}^ku_j(\cdot+w_{n,j})\to u$ in $X_{v}(\R^N)$;
\item[(iii)] $ {\mathcal L}_\lambda(u_n)\to {\mathcal L}_{\lambda}(u)+\sum_{j=1}^k {\mathcal I}(u_j)$.
	\end{enumerate}
\end{lemma}
\begin{proof}
As $(u_n)$ is a bounded sequence in $X_{v}(\R^N)$, there exists $u\in X_{v}(\R^N)$ such that, up to a subsequence, we have
\begin{equation}\label{firstconv}
\left\{
\begin{aligned}
u_n& \rightharpoonup u \quad\mbox{ weakly in }X_{v}(\R^N),\\
u_n &\rightharpoonup u\quad\mbox{ weakly in }L^s(\R^N),\; m_p\leq s\leq m_{p}^*,\\
u_n & \to u\quad\mbox{ a.e. in }\R^N.
\end{aligned}
\right.
\end{equation}
Using \eqref{firstconv} together with Lemma \ref{anbl}, we get  $${\mathcal L}_\lambda'(u)=0.$$ Hence, $u\in X_{v}(\R^N)$ is a solution of \eqref{fr}. Further, if $u_n\to u$ strongly in $X_{v}(\R^N)$ then $(A_1)$ holds and we are done.
	
\medskip
	
Next, let us assume that $(u_n)$ does not converge strongly to $u$ in $X_{v}(\R^N)$ and define $y_{n,1}=u_n-u$. Then $(y_{n,1})$ converges weakly (not strongly) to zero in $X_{v}(\R^N)$ and
\begin{equation}\label{bl2}
\|u_n\|_{X_{v}}^m=\|u\|_{X_{v}}^m+\|y_{n,1}\|_{X_{v}}^m+o(1).
\end{equation}
Also, by Lemma \ref{nlocbl} we have
\begin{equation}\label{bl3}
\int_{\R^N} \Big(|x|^{-\theta}*\frac{|u_n|^b}{|x|^\alpha}\Big)\frac{|u_n|^b}{|x|^\alpha}=\intr \Big(|x|^{-\theta}*\frac{|u|^b}{|x|^\alpha}\Big)\frac{|u|^b}{|x|^\alpha}+\intr \Big(|x|^{-\theta}*\frac{|y_{n, 1}|^b}{|x|^\alpha}\Big)\frac{|y_{n, 1}|^b}{|x|^\alpha}+o(1). 
\end{equation}
Using \eqref{bl2} and \eqref{bl3} we get
\begin{equation}\label{est6}
{\mathcal L}_\lambda(u_n)= {\mathcal L}_\lambda(u)+{\mathcal I}(y_{n,1})+o(1).
\end{equation}
Now, by Lemma \ref{anbl}, for any $h\in X_{v}(\R^{N})$, we have
\begin{equation}\label{est7}
\langle{\mathcal I}'(y_{n,1}), h\rangle=o(1).
\end{equation}
Further, using Lemma \ref{nlocbl} we get
$$
\begin{aligned}
0=\langle {\mathcal L}_\lambda'(u_n), u_n \rangle&=\langle {\mathcal L}_\lambda'(u),u\rangle+\langle {\mathcal I}'(y_{n,1}), y_{n,1} \rangle+o(1)\\
&=\langle{\mathcal I}'(y_{n,1}), y_{n,1}\rangle+o(1),
\end{aligned}
$$
which implies
\begin{equation}\label{est8}
\langle {\mathcal I}'(y_{n,1}), y_{n,1}\rangle=o(1).
\end{equation}
Next, we claim that 
$$
\Delta:=\limsup_{n\to \infty}\Big(\sup_{w\in \R^N} \int_{B_1(w)}|y_{n,1}|^{\frac{2Nb}{2N-2\alpha-\theta}}\Big)> 0.
$$
Let us assume that $\Delta= 0$. Using Lemma \ref{cc} we have $y_{n,1}\to 0$ strongly in $L^{\frac{2Nb}{2N-2\alpha-\theta}}(\R^N)$.  By double weighted Hardy-Littlewood-Sobolev inequality we get
$$
\int_{\R^N} \Big(|x|^{-\theta}*\frac{|y_{n,1}|^b}{|x|^\alpha}\Big)\frac{|y_{n,1}|^b}{|x|^\alpha}=o(1).
$$
Combining this together with \eqref{est8}, we deduce that $y_{n,1}\to 0$ strongly in $X_{v}(\R^{N})$, which is a contradiction. Hence, $\Delta> 0$.

As $\Delta>0$, one could find $w_{n,1}\in \R^N$ such that
\begin{equation}\label{est9}
\int_{B_1(w_{n,1})}|y_{n,1}|^{\frac{2Nb}{2N-2\alpha-\theta}}>\frac{\Delta}{2}.
\end{equation}
For the sequence $(y_{n,1}(\cdot+w_{n,1}))$, there exists $u_1\in X_{v}(\R^{N})$ such that, up to a subsequence, we have  
$$
\begin{aligned}
y_{n,1}(\cdot+w_{n,1})&\rightharpoonup u_1\quad\mbox{ weakly in } X_{v}(\R^{N}),\\
y_{n,1}(\cdot+w_{n,1})&\to u_1\quad\mbox{ strongly in } L_{loc}^{\frac{2Nb}{2N-2\alpha-\theta}}(\R^N),\\
y_{n,1}(\cdot+w_{n,1})&\to u_1\quad\mbox{ a.e. in } \R^N.
\end{aligned}
$$
Passing to the limit in \eqref{est9}, we have
$$
\int_{B_1(0)}|u_{1}|^{\frac{2Nb}{2N-2\alpha-\theta}}\geq \frac{\Delta}{2},
$$
hence, $u_1\not\equiv 0$. As $(y_{n,1})$ converges weakly to zero in $X_{v}(\R^{N})$, we get that $(w_{n,1})$ is unbounded. Therefore, passing to a subsequence, we could assume that $|w_{n,1}|\to \infty$. Using \eqref{est8}, we have ${\mathcal I}'(u_1)=0$, which further implies that $u_1$ is a nontrivial solution of \eqref{squ}.
Now, define
$$
y_{n,2}(x)=y_{n,1}(x)-u_1(x-w_{n,1}).
$$
Similarly as before, we get
$$
\|y_{n,1}\|^m=\|u_1\|^m+\|y_{n,2}\|^m+o(1).
$$
By Lemma \ref{nlocbl} we have
$$
\int_{\R^N} \Big(|x|^{-\theta}*\frac{|y_{n,1}|^b}{|x|^\alpha}\Big)\frac{|y_{n,1}|^b}{|x|^\alpha}=\int_{\R^N} \Big(|x|^{-\theta}*\frac{|u_1|^b}{|x|^\alpha}\Big)\frac{|u_1|^b}{|x|^\alpha}+\intr \Big(|x|^{-\theta}*\frac{|y_{n,2}|^b}{|x|^\alpha}\Big)\frac{|y_{n,2}|^b}{|x|^\alpha}+o(1). 
$$
Therefore,
$$
{\mathcal I}(y_{n,1})={\mathcal I}(u_1)+{\mathcal I}(y_{n,2})+o(1).
$$
By \eqref{est6} we get
$$
{\mathcal L}_\lambda (u_n)= {\mathcal L}_\lambda (u)+{\mathcal I}(u_1)+{\mathcal I}(y_{n,2})+o(1).
$$
Using the same approach as above, we get
$$
\langle {\mathcal I}'(y_{n,2}),h\rangle =o(1)\quad\mbox{ for any }h\in X_{v}(\R^{N})
$$
and
$$
\langle {\mathcal I}'(y_{n,2}), y_{n,2}\rangle =o(1).
$$
Now, if $(y_{n,2}) \to 0$ strongly, then we are done by taking $k=1$ in the Lemma \ref{compact}. Assume $y_{n,2}\rightharpoonup 0$ weakly (not strongly) in $X_{v}(\R^{N})$, then we could iterate the whole  process and in $k$ number of steps we find a set of sequences $(w_{n,j})\subset \R^N$, $1\leq j\leq k$ with 
$$
|w_{n,j}|\to \infty\quad\mbox{  and }\quad |w_{n,i}-w_{n,j}|\to \infty\quad\mbox{  as }\; n\to \infty, i\neq j
$$
and $k$ nontrivial solutions  $u_1$, $u_2$, $\dots$, $u_k\in X_{v}(\R^{N})$ of \eqref{squ} such that, by denoting 
$$
y_{n,j}(x):=y_{n,j-1}(x)-u_{j-1}(x-w_{n,j-1})\,, \quad 2\leq j\leq k,
$$ 
we get
$$
y_{n,j}(x+w_{n,j})\rightharpoonup u_j\quad\mbox{weakly in }\; X_{v}(\R^{N})
$$
and
$$
{\mathcal L}_\lambda(u_n)= {\mathcal L}_\lambda(u)+\sum_{j=1}^k {\mathcal I}(u_j)+{\mathcal I}(y_{n,k})+o(1).
$$
Now, as ${\mathcal L}_\lambda (u_n)$ is bounded and ${\mathcal I}(u_j)\geq d_{\mathcal I}$, one could iterate the process only a finite number of times and with this, we conclude our proof.
\end{proof}

\begin{cor}\label{corr1} 
Any $(PS)_c$ sequence of ${\mathcal L}_\lambda \! \mid_{{\mathcal N}_\lambda} $ is relatively compact for any $c\in (0,d_{\mathcal I})$ . 
\end{cor}
\begin{proof}
Let us assume that $(u_n)$ is a $(PS)_c$ sequence of ${\mathcal L}_\lambda \! \mid_{{\mathcal N}_\lambda}$. Then, by Lemma \ref{compact} we have ${\mathcal I}(u_j)\geq d_{\mathcal I}$ and upto a subsequence $u_n\to u$ strongly in $X_{v}(\R^{N})$ and hence, $u$ is a solution of \eqref{fr}. 
\end{proof}

\subsection{Completion of the Proof of Theorem $1.1$}

We need the following result in order to complete the proof of Theorem \ref{gstate}.

\begin{lemma}\label{flg}
$$
d_{\lambda}<d_{\mathcal I}.
$$
\end{lemma}
\begin{proof}
Let us assume that $P\in X_{v}(\R^{N})$ is a groundstate solution of \eqref{squ} and by \cite{B2012, CW2009} we know that such a groundstate solution exists. Let us denote by $tP$, the projection of $P$ on ${\mathcal N_\lambda}$, that is, $t=t(P)>0$ is the unique real number such that $tP\in {\mathcal N_\lambda}$. Since, $P\in {\mathcal N}_{\mathcal I}$ and $tP\in {\mathcal N_\lambda}$, we have
\begin{equation}\label{g1}
||P||^m= \int_{\R^N} \Big(|x|^{-\theta}*\frac{|P|^b}{|x|^\alpha}\Big)\frac{|P|^b}{|x|^\alpha}
\end{equation}
and
$$
t^m\|P\|^m=t^{2b}\int_{\R^N} \Big(|x|^{-\theta}*\frac{|P|^b}{|x|^\alpha}\Big)\frac{|P|^b}{|x|^\alpha}+ \lambda t^{2c}\int_{\R^N} \Big(|x|^{-\gamma}*\frac{|P|^c}{|x|^\beta}\Big)\frac{|P|^c}{|x|^\beta}.
$$
Therefore, we get $t<1$. Now,
	
\begin{equation*}
\begin{aligned}
d_\lambda \leq {\mathcal L}_\lambda(tP)&=\frac{1}{m}t^{m}\|P\|^{m}-\frac{1}{2b}t^{2b}\int_{\R^N} \Big(|x|^{-\theta}*\frac{|P|^b}{|x|^\alpha}\Big)\frac{|P|^b}{|x|^\alpha}- \frac{\lambda}{2c}t^{2c} \int_{\R^N} \Big(|x|^{-\gamma}*\frac{|P|^c}{|x|^\beta}\Big)\frac{|P|^c}{|x|^\beta}\\
&= \Big(\frac{t^{m}}{m}-\frac{t^{2b}}{2b}\Big)\|P\|^{m}-\frac{1}{2c}\Big(t^m||P||^m-t^{2b}\int_{\R^N} \Big(|x|^{-\theta}*\frac{|P|^b}{|x|^\alpha}\Big)\frac{|P|^b}{|x|^\alpha}\Big)\\
&= t^{m} \Big(\frac{1}{m}-\frac{1}{2c}\Big)\|P\|^{m}+t^{2b}\Big(\frac{1}{2c}-\frac{1}{2b}\Big)\|P\|^{m}\\
&< \Big(\frac{1}{m}-\frac{1}{2c}\Big)\|P\|^{m}+\Big(\frac{1}{2c}-\frac{1}{2b}\Big)\|P\|^{m}\\
&< \Big(\frac{1}{m}-\frac{1}{2b}\Big)\|P\|^{m} ={\mathcal I}(P)= d_{\mathcal I},
\end{aligned}
\end{equation*}
as required.
\end{proof}

Next, we use the Ekeland variational principle, that is, for any $n\geq 1$ there exists $(u_n) \in {\mathcal N}_\lambda$ such that
\begin{equation*}
\begin{aligned}
{\mathcal L}_\lambda(u_n)&\leq d_\lambda+\frac{1}{n} &&\quad\mbox{ for all } n\geq 1,\\
{\mathcal L}_\lambda(u_n)&\leq {\mathcal L}_\lambda(\tilde{u})+\frac{1}{n}\|\tilde{u}-u_n\| &&\quad\mbox{ for all } \tilde{u} \in {\mathcal N}_\lambda \;\;,n\geq 1.
\end{aligned}
\end{equation*}
Further, one could easily deduce that $(u_n) \in {\mathcal N}_\lambda$ is a $(PS)_{d_\lambda}$ sequence for ${\mathcal L}_\lambda$ on ${\mathcal N}_\lambda$. Then, by Lemma \ref{flg} and Corollary \ref{corr1} we have that up to a subsequence $u_n \to u$ strongly in $X_{v}(\R^{N})$ which is a groundstate of ${\mathcal L}_\lambda$.

\section{Proof of Theorem \ref{signc}}

In this section, we are concerned the existence of a least energy sign-changing solution of \eqref{fr}. 

\subsection{Proof of Theorem }

\begin{lemma}\label{frl1}
Let $N>m\geq 2$, $b>c>m$ and $\lambda \in \R$. There exists a unique pair $(\tau_0, \delta_0)\in (0, \infty)\times (0, \infty)$ ,for any $u \in  X_{v}(\R^{N})$ and $u^{\pm} \neq 0$, such that $\tau_0 u^{+}+\delta_0 u^{-} \in {\cal \overline{N}}_\lambda$. Also, if $u\in {\cal \overline{N}}_\lambda$ then for all $\tau$, $\delta \geq 0$ we have ${\mathcal L}_\lambda(u)\geq {\mathcal L}_\lambda(\tau u^{+}+\delta u^{-})$.
\end{lemma}

\begin{proof}
In order to prove this lemma, we follow the idea developed in \cite{VX2017}. Define the function $ \varphi: [0, \infty)\times [0, \infty)\rightarrow \R$ by
\begin{equation*}
\begin{aligned}
\varphi(\tau, \delta)&= {\mathcal L}_\lambda(\tau^{\frac{1}{2b}} u^{+}+\delta^{\frac{1}{2b}} u^{-})\\
&= \frac{\tau^{\frac{m}{2b}}}{m}\|u^{+}\|_{X_{v}}^{m}+\frac{\delta^{\frac{m}{2b}}}{m}\|u^{-}\|_{X_{v}}^{m}-\lambda\frac{\tau^{\frac{c}{b}}}{2c}\int_{\R^{N}}\Big(|x|^{-\gamma}*\frac{(u^{+})^{c}}{|x|^\beta}\Big)\frac{(u^{+})^{c}}{|x|^\beta}-\lambda\frac{\delta^{\frac{c}{b}}}{2c}\int_{\R^{N}}\Big(|x|^{-\gamma}*\frac{(u^{-})^{c}}{|x|^\beta}\Big)\frac{(u^{-})^{c}}{|x|^\beta}\\
&-\lambda\frac{\tau^{\frac{c}{2b}}\delta^{\frac{c}{2b}}}{2c}\int_{\R^{N}}\Big(|x|^{-\gamma}*\frac{(u^{+})^{c}}{|x|^\beta}\Big)\frac{(u^{-})^{c}}{|x|^\beta}-\frac{\tau}{2b}\int_{\R^{N}}\Big(|x|^{-\theta}*\frac{(u^{+})^{b}}{|x|^\alpha}\Big)\frac{(u^{+})^{b}}{|x|^\alpha}\\ &-\frac{\delta}{2b}\int_{\R^{N}}\Big(|x|^{-\theta}*\frac{(u^{-})^{b}}{|x|^\alpha}\Big)\frac{(u^{-})^{b}}{|x|^\alpha}-\frac{\tau^{\frac{1}{2}}\delta^{\frac{1}{2}}}{2b}\int_{\R^{N}}\Big(|x|^{-\theta}*\frac{(u^{+})^{b}}{|x|^\alpha}\Big)\frac{(u^{-})^{b}}{|x|^\alpha}.
\end{aligned}
\end{equation*}

One could observe that $\varphi$ is strictly concave. Hence, $\varphi$ has at most one maximum point. On the other hand we have
\begin{equation}\label{fr2}
\lim_{\tau \rightarrow \infty}\varphi(\tau, \delta)= -\infty \mbox{ for all }\delta \geq 0 \quad\mbox{ and } \quad\mbox{ } \lim_{\delta \rightarrow \infty}\varphi(\tau, \delta)= -\infty \mbox{ for all }\tau \geq 0,
\end{equation}
and one could easily check that
\begin{equation}\label{fr3}
\lim_{\tau \searrow 0}\frac{\partial{\varphi}}{\partial{\tau}}(\tau, \delta)= \infty \mbox{ for all }\delta> 0 \quad\mbox{ and }  \lim_{\delta \searrow 0}\frac{\partial{\varphi}}{\partial{\delta}}(\tau, \delta)= \infty \mbox{ for all }\tau> 0.
\end{equation}
Therefore, by \eqref{fr2} and \eqref{fr3} maximum cannot be achieved at the boundary. Hence, $\varphi$ has exactly one maximum point $(\tau_0, \delta_0)\in (0, \infty)\times (0, \infty)$.

\end{proof}

\medskip

Next, we divide our proof into two steps.

\medskip

\noindent \text{Step 1.}{\it \;\; The energy level $\overline{d_\lambda}>0$ is achieved by some $\sigma \in {\cal \overline{N}}_{\lambda}$.}

\medskip

Let us assume that $(u_n)\subset {\cal \overline{N}}_{\lambda}$ be a minimizing sequence for $\overline{d_\lambda}$. Note that 
\begin{equation*}
\begin{aligned}
{\mathcal L}_\lambda(u_{n})&= {\mathcal L}_\lambda(u_{n})-\frac{1}{2c}\langle {\mathcal L}_\lambda'(u_{n}), u_{n}\rangle \\
&= \Big(\frac{1}{m}-\frac{1}{2c}\Big)\|u_{n}\|_{X_{v}}^{m}+\Big(\frac{1}{2c}-\frac{1}{2b}\Big)\int_{\R^{N}}\Big(|x|^{-\theta}*\frac{|u|^{b}}{|x|^{\alpha}}\Big)\frac{|u|^{b}}{|x|^\alpha}\\
&\geq \Big(\frac{1}{m}-\frac{1}{2c}\Big)\|u_{n}\|_{X_{v}}^{m}\\
&\geq C\|u_{n}\|_{X_{v}}^{m},
\end{aligned}
\end{equation*}
for some positive constant $C_1>0$. Hence, for $C_{2}> 0$ we have 
$$
\|u_{n}\|_{X_{v}}^{m}\leq C_{2}{\mathcal L}_{\lambda}(u_{n})\leq M,
$$
that is, $(u_n)$ is bounded in $X_{v}(\R^{N})$. This further implies that $(u_{n}^{+})$ and $(u_{n}^{-})$ are also bounded in $X_{v}(\R^{N})$. Therefore, passing to a subsequence, there exists $u^{+}$, $u^{-}\in X_{v}(\R^{N})$ such that
$$
u_{n}^{+}\rightharpoonup u^{+} \mbox{ and } u_{n}^{-}\rightharpoonup u^{-} \quad\mbox{ weakly in } X_{v}(\R^{N}).
$$
As $b$, $c>m \geq 2$ satisfy \eqref{p} and \eqref{q} or \eqref{p1} and \eqref{q1}, we have that the embeddings $X_{v}(\R^{N})\hookrightarrow L^{\frac{2Nb}{2N-2\alpha-\theta}}(\R^{N})$ and  $X_{v}(\R^{N})\hookrightarrow L^{\frac{2Nc}{2N-2\beta-\gamma}}(\R^{N})$ are compact. Thus,
\begin{equation}\label{m1}
u_{n}^{\pm} \rightarrow u^{\pm} \quad\mbox{ strongly in } L^{\frac{2Nb}{2N-2\alpha-\theta}}(\R^{N}) \cap L^{\frac{2Nc}{2N-2\beta-\gamma}}(\R^{N}).
\end{equation}
Using the double weighted Hardy-Littlewood-Sobolev inequality, we have
\begin{equation*}
\begin{aligned}
C\Big(\|u_{n}^{\pm}\|_{L^{\frac{2Nb}{2N-2\alpha-\theta}}}^{m}+\|u_{n}^{\pm}\|_{L^{\frac{2Nc}{2N-2\beta-\gamma}}}^{m}\Big)&\leq \|u_{n}^{\pm}\|_{X_{v}}^{m}\\
&= \int_{\R^{N}}\Big(|x|^{-\theta}*\frac{|u_n|^{b}}{|x|^\alpha}\Big)\frac{|u_{n}^{\pm}|^{b}}{|x|^\alpha}+|\lambda|\int_{\R^{N}}\Big(|x|^{-\gamma}*\frac{|u_n|^{c}}{|x|^\beta}\Big)\frac{|u_{n}^{\pm}|^{c}}{|x|^\beta}\\
&\leq C\Big(\|u_{n}^{\pm}\|_{L^{\frac{2Nb}{2N-2\alpha-\theta}}}^{b}+\|u_{n}^{\pm}\|_{L^{\frac{2Nc}{2N-2\beta-\gamma}}}^{c}\Big)\\
&\leq C\Big(\|u_{n}^{\pm}\|_{L^{\frac{2Nb}{2N-2\alpha-\theta}}}^{m}+\|u_{n}^{\pm}\|_{L^{\frac{2Nc}{2N-2\beta-\gamma}}}^{m}\Big)\Big(\|u_{n}^{\pm}\|_{L^{\frac{2Nb}{2N-2\alpha-\theta}}}^{b-m}+||u_{n}^{\pm}||_{L^{\frac{2Nc}{2N-2\beta-\gamma}}}^{c-m}\Big).
\end{aligned}
\end{equation*}
As $u_{n}^{\pm}\neq 0$, we get
\begin{equation}\label{m2}
\|u_{n}^{\pm}\|_{L^{\frac{2Nb}{2N-2\alpha-\theta}}}^{b-m}+\|u_{n}^{\pm}\|_{L^{\frac{2Nc}{2N-2\beta-\gamma}}}^{c-m}\geq C> 0 \quad\mbox{ for all } n\geq 1.
\end{equation}
Therefore, using \eqref{m1} and \eqref{m2} one could have that $u^{\pm} \neq 0$. Next, using \eqref{m1} together with double weighted Hardy-Littlewood-Sobolev inequality, we deduce
\begin{equation*}
\begin{aligned}
&\int_{\R^{N}}\Big(|x|^{-\theta}*\frac{(u_{n}^{\pm})^{b}}{|x|^\alpha}\Big)\frac{(u_{n}^{\pm})^{b}}{|x|^\alpha} &&\rightarrow \int_{\R^{N}}\Big(|x|^{-\theta}*\frac{(u^{\pm})^{b}}{|x|^\alpha}\Big)\frac{(u^{\pm})^{b}}{|x|^\alpha},\\
&\int_{\R^{N}}\Big(|x|^{-\theta}*\frac{(u_{n}^{+})^{b}}{|x|^\alpha}\Big)\frac{(u_{n}^{-})^{b}}{|x|^\alpha} &&\rightarrow \int_{\R^{N}}\Big(|x|^{-\theta}*\frac{(u^{+})^{b}}{|x|^\alpha}\Big)\frac{(u^{-})^{b}}{|x|^\alpha},\\
&\int_{\R^{N}}\Big(|x|^{-\gamma}*\frac{(u_n^{\pm})^{c}}{|x|^\beta}\Big)\frac{(u_n^{\pm})^{c}}{|x|^\beta} &&\rightarrow \int_{\R^{N}}\Big(|x|^{-\gamma}*\frac{(u^{\pm})^{c}}{|x|^\beta}\Big)\frac{(u^{\pm})^{c}}{|x|^\beta},
\end{aligned}
\end{equation*}
and
\begin{equation*}
\int_{\R^{N}}\Big(|x|^{-\gamma}*\frac{(u_n^{+})^{c}}{|x|^\beta}\Big)\frac{(u_n^{-})^{c}}{|x|^\beta}\rightarrow \int_{\R^{N}}\Big(|x|^{-\gamma}*\frac{(u^{+})^{c}}{|x|^\beta}\Big)\frac{(u^{-})^{c}}{|x|^\beta}.
\end{equation*}
Next, by using Lemma \ref{frl1}, we get that there exists a unique pair $(\tau_{0}, \delta_{0})$ such that $\tau_{0} u^{+}+\delta_{0} u^{-}\in {\cal \overline{N}}_{\lambda}$. Further, using the fact that the norm $\|.\|_{X_{v}}$ is weakly lower semi-continuous, we get
\begin{equation*}
\begin{aligned}
\overline{d_\lambda} \leq {\mathcal L}_\lambda(\tau_{0} u^{+}+\delta_{0} u^{-})&\leq \liminf_{n\rightarrow \infty} {\mathcal L}_\lambda(\tau_{0} u^{+}+\delta_{0} u^{-})\\
&\leq \limsup_{n\rightarrow \infty} {\mathcal L}_\lambda(\tau_{0} u^{+}+\delta_{0} u^{-})\\
&\leq \lim_{n\rightarrow \infty}{\mathcal L}_\lambda(u_{n})\\
&= \overline{d_\lambda}.
\end{aligned}
\end{equation*}
We conclude by taking $\sigma= \tau_{0} u^{+}+\delta_{0} u^{-}\in {\cal \overline{N}}_\lambda$.

\medskip

\noindent \text{Step 2.}{ \it \;\;  ${\mathcal L}_\lambda'(\sigma)=0$, that is, $\sigma \in {\cal \overline{N}}_\lambda$ is the critical point of ${\mathcal L}_\lambda:X_{v}(\R^{N}) \rightarrow \R$.
}

\medskip

Say $\sigma$ is not a critical point of ${\mathcal L}_{\lambda}$, then there exists $\kappa\in C_{c}^{\infty}(\R^N)$ such that $ \langle {\mathcal L}_\lambda'(\sigma), \kappa \rangle= -2.$ As ${\mathcal L}_{\lambda}$ is continuous and differentiable, so there exists $\zeta>0$ small such that
\begin{equation}\label{fr4}
\langle {\mathcal L}_\lambda'(\tau u^{+}+\delta u^{-}+\omega \bar{\sigma}), \bar{\sigma} \rangle \;\; \leq -1 \quad\mbox{ if } (\tau- \tau_{0})^{2}+(\delta- \delta_0)^{2}\leq \zeta^{2} \mbox{ and } 0\leq \omega \leq \zeta.
\end{equation}
Next, let us asumme that $D\subset \R^{2} $ is an open disc of radius $\zeta>0$ centered at $(\tau_0, \delta_0)$ and define a continuous function $\varPhi: D\rightarrow [0, 1]$ by 
\begin{equation*}
\varPhi(\tau, \delta)= \left\{\begin{array}{cc}1\quad\mbox{ if }(\tau- \tau_{0})^{2}+(\delta- \delta_0)^{2}\leq \frac{\zeta^{2}}{16}, \\ 0\quad\mbox{ if }(\tau- \tau_{0})^{2}+(\delta- \delta_0)^{2}\geq \frac{\zeta^{2}}{4}.\end{array} \right.
\end{equation*}
Also, let us define a continuous map $T: D\rightarrow X_{v}(\R^{N})$ as 
\begin{equation*}
T(\tau, \delta)= \tau u^{+}+\delta u^{-}+\zeta \varPhi(\tau, \delta)\bar{\sigma} \quad\mbox{ for all } (\tau, \delta)\in D
\end{equation*}
and $Q: D\rightarrow \R^{2}$ as
\begin{equation*}
Q(\tau, \delta)= (\langle {\mathcal L}_\lambda'(T(\tau, \delta)), T(\tau, \delta)^{+}\rangle, \langle {\mathcal L}_\lambda'(T(\tau, \delta)), T(\tau, \delta)^{-}\rangle) \quad\mbox{ for all }(\tau, \delta)\in D. 
\end{equation*}
As the mapping $u \mapsto u^{+}$ is continuous in $X_{v}(\R^{N})$, we get that $Q$ is also continuous. Furthermore, if we are on the boundary of $D$, that is,   $(\tau- \tau_{0})^{2}+(\delta- \delta_0)^{2}= \zeta^{2}$, then $\varPhi= 0$ according to the definition. Therefore, we get  $T(\tau, \delta)= \tau u^{+}+\delta u^{-}$ and by Lemma \ref{frl1}, we deduce 
\begin{equation*}
Q(\tau, \delta)\neq 0 \quad\mbox{ on } \partial{D}.
\end{equation*}
Hence, the Brouwer degree is well defined and $\deg(Q, {\rm int} (D), (0, 0))=1$ and there exists $(\tau_1, \delta_1)\in {\rm int} (D)$ such that $Q(\tau_1, \delta_1)= (0, 0)$. Therefore, we get that $T(\tau_1, \delta_1)\in {\cal \overline{N}}_{\lambda}$ and by the definition of $\overline{d_\lambda}$ we deduce that
\begin{equation}\label{fr5}
{\mathcal L}_\lambda(T(\tau_1, \delta_1))\geq \overline{d_\lambda}.
\end{equation}
Next, by equation \eqref{fr4}, we have
\begin{equation}\label{fr6}
\begin{aligned}
{\mathcal L}_\lambda(T(\tau_1, \delta_1))&= {\mathcal L}_\lambda(\tau_1 u^{+}+\delta_{1} u^{-})+\int_{0}^{1}\frac{d}{dt}{\mathcal L}_\lambda(\tau_1 u^{+}+\delta_{1} u^{-}+\zeta t \varPhi(\tau_1, \delta_1)\bar{\sigma})dt \\
&= {\mathcal L}_\lambda(\tau_1 u^{+}+\delta_{1} u^{-})-\zeta \varPhi(\tau_1, \delta_1).
\end{aligned}
\end{equation}
Now, by definition of $\varPhi$ we have $\varPhi(\tau_1, \theta_1)=1$ when $(\tau_1, \delta_1)=(\tau_0, \delta_0)$. Hence, we deduce that
$$
{\mathcal L}_\lambda(T(\tau_1, \delta_1))\leq {\mathcal L}_\lambda(\tau_1 u^{+}+\delta_{1} u^{-})-\zeta\leq \overline{d_\lambda}-\zeta< \overline{d_\lambda}.
$$
The case when $(\tau_1, \delta_1)\neq (\tau_0, \delta_0)$, then by Lemma \ref{frl1} we have 
$$
{\mathcal L}_\lambda(\tau_1 u^{+}+\delta_{1} u^{-})< {\mathcal L}_\lambda(\tau_0 u^{+}+\delta_{0} u^{-})= \overline{d_\lambda},
$$
which further gives
$$
{\mathcal L}_\lambda(T(\tau_1, \delta_1))\leq {\mathcal L}_\lambda(\tau_1 u^{+}+\delta_{1} u^{-})< \overline{d_\lambda}.
$$
This contradicts the equation \eqref{fr5} and with this we conclude our proof. 

\medskip

\section*{Acknowledgements}
The author acknowledges the financial support of Irish Research Council Postdoctoral Scholarship number R13165.

\end{document}